\pdfoutput=1 
\documentclass[a4paper]{article}
\usepackage{amsthm,amssymb,amsmath,enumerate,graphicx,tikz-cd}
\usepackage[capitalize]{cleveref}
\usepackage[nocompress]{cite}

\newcommand\COMMENT[1]{}

\newdimen\margin
\def\textno#1&#2\par{%
	\margin=\hsize
	\advance\margin by -4\parindent
	\setbox1=\hbox{\sl#1}%
	\ifdim\wd1 < \margin
	$$\box1\eqno#2$$\endgraf%
	\else
	\bigbreak
	\line{\indent$\vcenter{\advance\hsize by -3\parindent
			\sl\noindent#1}\hfil#2$}%
	\bigbreak
	\fi}
\newtheorem{THM}{Theorem} 
\newtheorem{LEM}[THM]{Lemma}
\newtheorem{COR}[THM]{Corollary}

\newtheorem{PROP}[THM]{Proposition}

\newtheorem{PROB}[THM]{Problem}
\newtheorem{CLAIM}{Claim}

\theoremstyle{definition}
\newtheorem{EX}[THM]{Example}

\def\specrelabove#1#2{\mathrel{\mathop{\kern0pt #1}\limits^{#2}}}
\def\specrelbelow#1#2{\mathrel{\mathop{\kern0pt #1}\limits_{#2}}}

\def\F{\mathcal F}

\newcommand\N{\mathbb N}

\newcommand\R{\mathbb R}
\newcommand\Z{\mathbb Z}

\newcommand{\Max}{\mu}

\def\lowfwd #1#2#3{{\mathop{\kern0pt #1}\limits^{\kern#2pt\raise.#3ex \vbox to 0pt{\hbox{$\scriptscriptstyle\rightarrow$}\vss}}}}

\def\lowbkwd #1#2#3{{\mathop{\kern0pt #1}\limits^{\kern#2pt\raise.#3ex
			\vbox to 0pt{\hbox{$\scriptscriptstyle\leftarrow$}\vss}}}}

\def\vr{\lowfwd r{1.5}2}
\def\rv{\lowbkwd r{1.5}2}
\def\vS{\vec S}

\def\vSdash{{\mathop{\kern0pt S\lower-1pt\hbox{${}
				\scriptstyle'$}}\limits^{\kern2pt\raise.1ex
			\vbox to 0pt{\hbox{$\scriptscriptstyle\rightarrow$}\vss}}}}
\def\vs{\lowfwd s{1.5}2}
\def\sv{\lowbkwd s{1.5}2}
\def\vsdash{{\mathop{\kern0pt s\lower.5pt\hbox{${}
				\scriptstyle'$}}\limits^{\kern0pt\raise.02ex
			\vbox to 0pt{\hbox{$\scriptscriptstyle\rightarrow$}\vss}}}}
\def\svdash{{\mathop{\kern0pt s\lower.5pt\hbox{${}
				\scriptstyle'$}}\limits^{\kern0pt\raise.02ex
			\vbox to 0pt{\hbox{$\scriptscriptstyle\leftarrow$}\vss}}}}

\def\vsi{\lowfwd {s_i}11}

\def\vsj{\lowfwd {s_j}01}

\def\vsell{\lowfwd {s_\ell}11}

\def\vsik{\lowfwd {s_{i_k}}{-2}2}
\def\vsione{\lowfwd {s_{i_1}}{-2}2}

\def\vt{\lowfwd t{1.5}2}
\def\tv{\lowbkwd t{1.5}2}

\def\vU{\vec U}

\def\vxdash{{\mathop{\kern0pt x\lower.5pt\hbox{${}
				\scriptstyle'$}}\limits^{\kern0pt\raise.02ex
			\vbox to 0pt{\hbox{$\scriptscriptstyle\rightarrow$}\vss}}}}
\def\xvdash{{\mathop{\kern0pt x\lower.5pt\hbox{${}
				\scriptstyle'$}}\limits^{\kern0pt\raise.02ex
			\vbox to 0pt{\hbox{$\scriptscriptstyle\leftarrow$}\vss}}}}

\def\es{\emptyset}
\def\sub{\subseteq}

\def\sm{\smallsetminus}

\title{Point sets and functions inducing tangles\\ of set separations}
\author{Reinhard Diestel \and Christian Elbracht \and Raphael~W.\ Jacobs}
\date{}

\begin{document}
	\abovedisplayshortskip=-3pt plus3pt
	\belowdisplayshortskip=6pt
	
	\maketitle

	\begin{abstract} \noindent
		Tangles, as introduced by Robertson and Seymour, were designed as an indirect way of capturing clusters in graphs and matroids.
		They have since been shown to capture clusters in much broader discrete structures too.
		But not all tangles are induced by a set of points, let alone a cluster.
		We characterise those that are: the tangles that are induced by a subset of or function on the set of data points whose connectivity structure they are meant to capture.
		
		We offer two such characterisations.
		The first is in terms of how many small sides of a tangle's separations it takes to cover the ground set.
		The second uses a new notion of duality for oriented set separations that becomes possible if these are no longer required to be separations of graph or matroids.
	\end{abstract}

	\section{Introduction} \label{sec:Introduction}
	
	Tangles were introduced by Robertson and Seymour as a tool in their graph minors project~\cite{GMX}.
	They provided a novel, indirect, way to capture highly cohesive substructures, or `clusters', in graphs.
	The idea is that since clusters cannot be divided into significantly large parts by graph separations of low order, any given cluster implicitly orients every low-order separation towards its `big' side, the side that contains most of the cluster.
	It turned out that this induced orientation of all the low-order separations collectively contains all the information needed to prove fundamental theorems about the cluster structure of a graph, which has made tangles a powerful tool in the connectivity theory of graphs.\looseness=-1
	
	Over the last decade, the theory of tangles has been significantly generalised to other discrete structures~\cite{SARefiningInessParts, AbstractTangles, CanonicalToTSubmodular, FiniteSplinters, InfiniteSplinters, ToTfromTTD}.
	These include matroids~\cite{BranchDecMatroids}, but also bespoke structures that come with concrete clustering applications~\cite{TangleBook, TangleClusteringWeakStrong}.
	This has been made possible by re-casting the notion and theory of tangles in terms of a purely algebraic framework of `separation systems'~\cite{ASS, ProfilesNew, TangleTreeAbstract}, which encompass the notions of separation from all these various different contexts. 
	Although these separation systems are very general, the central tangle theorems are still valid in this framework.
	
	In this paper we are concerned with only the most basic type of separation systems, those of sets.
	A~\emph{separation} of a set~$V\!$ is an unordered pair~$s=\{A,B\}$ of subsets of~$V\!$ such that~$A\cup B = V\!$.
	It has two \emph{orientations}: the ordered pair~$(A,B)$, which we think of as pointing towards~$B$, and its \emph{inverse}~$(B,A)$.
	We usually denote the two orientations of a separation~$s$ by arrows: one of them is denoted as~$\vs$, the other as~$\sv$, but it does not matter which is which.
	
	Now consider a particular set~$S$ of separations of a set~$V\!$, and a subset~$X\sub V\!$ that is a `cluster' in the sense that the separations in~$S$ cannot divide it evenly: let us assume that for every~$\{A,B\}\in S$ more than two thirds of~$X$ lies in~$A\sm B$ or in~$B\sm A$.
	If $X$ has more elements in~$B$ than in~$A$, we think of $X$ as orienting this separation towards~$B$; in our notation it orients it `as~$(A,B)$'.
	
	Any orientation~$\tau$ of (all the separations in)~$S$ induced by a cluster~${X \subseteq V}$ in this way has the following property, which no longer refers to~$X$: whenever~$\tau$ orients three separations~${\{A_i,B_i\}\in S}$ ($i=1,2,3$) towards~$B_i$, their `small sides'~$A_i$ cannot cover~$V\!$, because each contains less than a third of~$X$.
	This is the most basic definition of a \emph{tangle} of~$S$: any orientation of~$S$ such that no three small sides cover~$V\!$.
	Note that the meaning of `small' here is intrinsic to~$\tau$: the side of a separation towards which~$\tau$ orients it is now called~`big', its other side `small'.
	No reference to a `cluster'~$X$ is made in this definition of a tangle.
	
	This abstract definition of a tangle has made it possible to investigate clusters in a graph or data set without referring to them directly in the usual concrete way, as sets of vertices or data points. In particular, one can investigate the relative structure of clusters without even having found them in this concrete sense~-- a~sense which, moreover, may be inadequate given the fuzzy nature of many real-world clusters, which often do not enable us to decide easily which data points belong to a cluster and which do not.
	
	However, tangles are not equivalent to clusters but more general: while every cluster, in our earlier informal sense, gives rise to a tangle, not every tangle is induced by such a cluster. We shall see an example in a moment.%
	\COMMENT{}
	Tangles that do not come from clusters can still be interesting; the \emph{text tangles} in~\cite{TangleBook,TextTangles} are a typical example in the context of set separations.%
	\COMMENT{}
	
	To be a little more formal, let us say that a set~$X\sub V\!$ of `points' \emph{induces} an orientation~$\tau$ of~$S$ if, for every~$(A,B)\in\tau$, there are more elements of~$X$ in~$B$ than in~$A$.
	It is an open question whether every tangle of a graph is induced by some set of its vertices.%
	\footnote{It was shown in~\cite{weighted_deciders_AIC} that graph tangles are induced by sets of vertices with weights assigned to them; we then say that these weight functions induce those tangles; see later.}%
	\COMMENT{}
	Tangles of more general set separations, however, need not be induced by a set of points.
	Let us construct a simple example.%
	\COMMENT{}
	\COMMENT{}

	
	\medbreak
	
	The basic idea of our construction is that we start with~$S$ and an `orientation' ${\tau = \{\vs\mid s\in S\,\}}$ of~$S$ as just a collection of names. Every~$\vs$ will eventually become a pair~$(A,B)$ of subsets of some set~$V\!$ to be constructed; a pair of subsets that form a separation of~$V\!$. In order to flesh out our notational shells for ${\tau = \{\vs\mid s\in S\}}$ and each $\vs=(A,B)$ in a way that makes~$\tau$ into a tangle of~$S$ not induced by any subset of~$V\!$, we build~$V\!$ element by element, immediately assigning every~$v\in V\!$ to either~$A$ or~$B$ for every~$\vs = (A,B)\in\tau$.
	When we are done, we shall verify that~$\tau$ is indeed a tangle of~$S$ not induced by any set in~$V\!$.
	
	To implement this plan, we create for every $3$-set~$\sigma \subseteq \tau$ one element~$v_\sigma$ for~$V$ so that these~$v_\sigma$ are distinct for different~$\sigma$, and we let~$V$ be the set of all these~$v_\sigma$.
	For each~$(A,B) \in \sigma$ we put~$v_\sigma$ in~$B$, and for each~$(A,B) \in \tau \sm \sigma$ we put~$v_\sigma$ in~$A$.
	Then for every~$\vs = (A,B)\in\tau$ we have~$B = \{\,v_\sigma \mid \vs\in\sigma\}$ and~$A = V \sm B$.
	Now if~$|S| = m$, say, then~$|V| = \binom{m}{3}$, and for every~$(A,B)\in \tau$ we have~$|B| = \binom{m-1}{2}$.
	Let us choose~$S$ so that~$m \ge 6$.
	
	By construction,~$\tau = \{\vs\mid s\in S\,\}$ is a tangle of~$S$: the big sides of any three\,$\vs\in\tau$ have an element in common.
	Conversely,
	\begin{equation*} \label{equ:star}
		\textit{every~$v\in V\!$ lies on the big side of exactly three elements of~$\tau$.}
		\tag{$*$}
	\end{equation*}
	\COMMENT{}
	
	A simple double count now shows that~$\tau$ is not induced by any set~$X\sub V\!$.
	Indeed, suppose it is and let
	\begin{equation*}
		d := \sum_{(A,B)\in\tau} |X\cap B|.
	\end{equation*}
	Then~$|X\cap B| > |X|/2$ for every~$(A,B)\in\tau$, because $X$ induces~$\tau$, and hence $d > m\,|X|/2$.
	On the other hand, by~$\eqref{equ:star}$, each~$x\in X$ lies in~$B$ for exactly three~$(A,B)\in \tau$, so~$d$ counts it three times:~$d = 3\,|X|$.
	Putting these together we obtain~$3\,|X| > m\,|X|/2$.
	This implies~$m<6$, contrary to our assumption.
	
	\medbreak
	
	Is it possible to distil from this example some property of~$\tau$ that identifies \emph{all} the tangles that are not induced by any set of `points', elements of their ground set?
	We grappled with this question for quite a while, until we found the following solution.
	
	Given an integer~$k$, we say that an orientation~$\tau$ of~$S$ is \emph{$k$-resilient} if it takes more than~$k$ elements of~$\tau$ to obtain~$V\!$ as the union of their small sides.%
	\COMMENT{}
	Every tangle, by definition, is~$3$-resilient.
	
	Our earlier example of a tangle~$\tau$ not induced by a set in~$V\!$ is not \hbox{$4$-resilient}.
	In~fact, given any%
	\COMMENT{}
	four distinct separations in~$\tau$, by~$\eqref{equ:star}$ every~$v\in V\!$ lies on the small side of at least one of them, so the four small sides have union~$V\!$.
	At the other extreme, every \emph{principal} tangle of a set of bipartitions of a set~$V\!$, one consisting of all bipartitions~$(A,B)$ of~$V\!$ whose big side~$B$ contains some fixed element~$x \in V\!$, is \emph{infinitely resilient} in that it is \hbox{$k$-resilient} for every~$k \in \N$.
	Note that~$\{x\}$ induces this tangle.
	In \cref{sec:Resilience} we shall see that the unique~$5$-tangle of the~$(n\times n)$-grid,	which is induced by its entire vertex set, is~$\Omega(n^2)$-resilient.
	
	These examples seem to suggest that tangles of set separations that are \hbox{$k$-resilient} for large~$k$ are more likely to be induced by subsets of their ground set.
	We can indeed prove such a fact, with an interesting additional twist:
	`large' has to be measured not in terms of~$|V|$ or~$|S|$, but relative to the number of maximal elements of the tangle in the usual partial order of oriented separations (see \cref{sec:Prelims}).
	This dependence on the number of maximal elements in a tangle is quite natural:
	in a~$k$-resilient tangle with at most~$k$ maximal elements, the intersection of all their big sides is non-empty, and it clearly induces this tangle.
	
	Let us say that a function~$w\colon V\!\to\R_{\ge 0}$, which we may think of as placing weights on the elements of~$V\!$, \emph{orients} a separation $s = \{A,B\}$ of~$V\!$ \emph{as} $\vs = (A,B)$ if $w(A) < w(B)$, where $w(A) = \sum_{v\in A} w(v)$ and likewise for~$B$.
	More generally, $w$~\emph{orients} a set~$S$ of separations of~$V\!$ \emph{as}~$\{\vs\mid s\in S\}$ if it orients every $s\in S$ as~$\vs$.\looseness=-1
	
	Conversely, if~$\tau$ is a given orientation of~$S$ and ${w\colon V\!\to\R_{\ge 0}}$ orients~$S$ as~$\tau$, we say that $w$ {\em induces\/}~$\tau$. If~$w$ induces~$\tau$ and takes values in~$\{0,1\}$, we likewise say that $X = w^{-1}(1)\sub V\!$ {\em induces\/}~$\tau$;  note that this agrees with our earlier informal notion of tangle-inducing subsets of~$V\!$.
	
	Our earlier observation that $k$-resilient tangles with at most~$k$ maximal elements are induced by point sets can be strengthened for inducing functions:
	
	\begin{THM} \label{thm:ResilienceWeightedInformal}
		A tangle with~$m$ maximal elements is induced by some function if it is~$k$-resilient for some $k > \frac{m}{2}$.
	\end{THM}
	
	\noindent
	We shall see that, as a general bound for all tangles, this is best possible.
	
	For our proof of \cref{thm:ResilienceWeightedInformal} we introduce the notion of being `locally induced', which generalises the idea of resilience.
	We show in \cref{thm:DecidableWeighted} that an orientation~$\tau$ of a	set~$S$ of separations is induced by a function if and only if there exist suitable parameters~$k$ and~$\ell$ such that~$\tau$ is \hbox{`$k$-locally $\ell$-induced'}.
	In particular, there are such suitable parameters~$k$ and~$\ell$ if~$\tau$ is highly resilient compared with its number of maximal elements, our \cref{thm:ResilienceWeightedInformal} above.
	
	For our second characterisation of tangles induced by functions%
	\COMMENT{}
	we exploit the recent notion of \emph{duality} of sets of separations.
	Duality between oriented set separations naturally arises in applications of tangle theory.
	For example, consider an online shop with a set~$V\!$ of items on sale and a history~$P$ of purchases made last year~\cite{TanglesSocial, TangleBook}.
	In this setting, two different sets of separations occur.
	Every purchase in~$P$ induces a bipartition of~$V\!$ into the items bought versus those not bought.
	Equally, every item in~$V\!$ defines a bipartition of~$P$ into those purchases that included it versus those that did not.
	The tangles of these two sets of separations can be shown to interact~\cite{ASSduality}, and they will help us to obtain a second characterisation of tangles induced by functions, \cref{thm:WeightedLPDuality}.
	This will also imply \cref{thm:ResilienceWeightedInformal}.
	
	In our third contribution in this paper we show that some tangles are induced by point sets if the separations they orient are endowed with an \emph{order function}, a~function that assigns an integer ${|s| =: |A,B|\ge 0}$ to every separation ${s = \{A,B\} \in U}$ where~$U = U(V)$ is the set of all separations of a set~$V\!$ (see \cref{sec:Prelims} for the precise definitions). Given such an order function on~$U$ and an integer~$k \ge 1$, the {\em$k$-tangles\/} in~$U$ are the tangles of its subset~$S_k := \{\,s\in U : |s| < k\,\}$.
	
	Elbracht, Kneip, and Teegen~\cite{weighted_deciders_AIC} showed that the~$k$-tangles in~$U$ are induced by functions on~$V\!$, though not necessarily by subsets of~$V\!$, when their order is defined as $|A,B| := |A \cap B|$.%
	\COMMENT{}
	In \cref{sec:Extendable} we strengthen their result for~$k$-tangles in~$U$ which extend to~$2k$-tangles in~$U$, by showing that such~$k$-tangles are even induced by sets.
	It would be interesting to know whether similar results hold for other submodular (see~\cref{sec:Prelims}) order functions on~$U$ than the above.

	\section{Preliminaries} \label{sec:Prelims}
	
	This section collects together the definitions we need in this paper.
	While we shall work only with separations of sets as considered in the introduction, all definitions given here fit into the more general framework of `abstract separation systems'~\cite{ASS}.
	From this framework we shall borrow some notations which we will also introduce in what follows.
	
	In this paper we consider separations of finite sets~$V$, this finiteness assumption on~$V$ will not be mentioned explicitly in the remainder of this paper.
	Throughout we will consider various finite sets~$V$; if~$V$ is not specified explicitly, then~$V$ is any arbitrary finite set.
	
	For definitions around graphs we refer the reader to \cite{DiestelBook16noEE}.
	For every~$k \in \N$ we write~$[k] := \{1, \dots, k\}$, we call a set with~$k$ elements a~\emph{$k$-set}, and we denote the set of all~$k$-element subsets of a set~$X$ as~$X^{(k)}$.

	\subsection{Separations of sets}
	An \emph{(unoriented) separation} of~$V$ is an unordered pair~$\{A, B\}$ of subsets~$A$ and~$B$ of~$V$, its \emph{sides}, such that~$A \cup B = V$.
	The two \emph{orientations} of~$\{A,B\}$ are the ordered pairs~$(A, B)$ and~$(B, A)$ which is the \emph{inverse} of~$(A,B)$.\footnote{Note that~$(A,B)$ and~$(B,A)$ coincide if and only if~$A = B = V$. In particular, the oriented separation~$(V,V)$ equals its own inverse.}
	We write~$U(V)$ for the set of all unoriented separations of~$V$.
	
	Every ordered pair~$(A,B)$ of subsets of~$V$ with~$A \cup B = V$ is an \emph{oriented separation} of~$V$.
	Its \emph{underlying} unoriented separation is~$\{A,B\}$, and~$(B, A)$ is its \emph{inverse}.
	Given an oriented separation~$(A,B)$ of~$V$, we refer to~$A$ as its \emph{small side} and call~$B$ its \emph{big side}.
	We shall informally use the term `separation' also as a short term for oriented separations, but we will only do so if the meaning is unambiguous.
	
	As indicated in the introduction, we fix the following notational conventions for separations for better readability:
	unoriented separations will be denoted as lowercase letters, such as~$s$.
	Given an unoriented separation~$s$ of a set, we denote its two orientations as~$\vs$ and~$\sv$.
	There is no default orientation: once we have called one of the two orientations~$\vs$, the other one will be~$\sv$, and vice-versa. 
	Oriented separations will be denoted as lowercase letters with a forward or backward arrow on top, such as~$\vs$ and~$\sv$.
	Given an oriented separation~$\vs$ of~$V$, its underlying unoriented separation is denoted as~$s$, and its inverse as~$\sv$.
	
	We define a partial order~$\le$ on the set of oriented separations of~$V$ as follows:
	for two oriented separations~$(A,B)$ and~$(C,D)$ of~$V\!$, we let ${(A,B) \le (C,D)}$ if~$A \subseteq C$ and~$B \supseteq D$; we write~$(A,B) < (C,D)$ if and only if~$(A,B) \le (C,D)$ and~$(A,B) \neq (C,D)$.
	With this definition we in particular have 
	\begin{equation*}
		(A,B) \le (C,D) \iff (B,A) \ge (D,C).
	\end{equation*}
	The \emph{maximal elements} of a set~$\sigma$ of oriented separations of~$V$ are always those separations in~$\sigma$ which are maximal with respect to this partial order~$\le$.
	
	Given two oriented separations $\vr = (A,B)$ and $\vs = (C,D)$ of~$V\!$, their supremum $\vr \vee \vs$ in~$U(V)$ with respect to the partial order~$\le$ is the oriented separation $(A \cup C, B \cap D)$, and their infimum $\vr \wedge \vs$ in~$U(V)$ is ${(A \cap C, B \cup D)}$.
	Note that the supremum and the infimum satisfy DeMorgan's law in that $\rv \vee \sv$ is the inverse of~$\vr \wedge \vs$ for every two oriented separations~$\vr$ and~$\vs$ of~$V$.
	
	A set~$U$ of unoriented separations of~$V$ is a \emph{universe of separations of~$V$} if the set~$\vU := \{ \vs, \sv \mid s \in U\}$ of all orientations of separations in~$U$ is closed under taking suprema and infima in~$U(V)$.
	Note that $U(V)$ itself is a universe of separations of~$V\!$, by definition.
	
	An oriented separation~$\vs = (A,B)$ of~$V$ is \emph{small} if~$B = V$, and \emph{co-small} if~$A = V$; thus,~$\vs$ is small if and only if its inverse~$\sv$ is co-small.
	Note that~$\vs$ is small if and only if~$\vs \le \sv$.
	If~$A = V = B$ or, equivalently, if~$\vs = \sv$, then both~$\vs$ and~$s$ are called \emph{degenerate}; otherwise,~$\vs$ and~$s$ are \emph{non-degenerate}.
	A~set~$\sigma$ of oriented non-degenerate separations of~$V$ is a \emph{star} if~$\vr \le \sv$ for every two distinct~$\vr, \vs \in \sigma$.
	
	Given a universe~$U$ of separations of~$V$, an \emph{order function}~$|\cdot|$ on~$U$ assigns to each~$s \in U$ its \emph{order}~$|s| \in \Z_{\ge 0}$; the order of an orientation~$\vs$ of a separation~${s \in U}$ is defined to be the order of~$s$.
	Given~$k \in \N$ we write~$S_k$ for the set of all separations in~$U$ of order less than~$k$.
	
	An order function on~$U$ is \emph{submodular} if, for every two~$r, s \in U$ and arbitrary orientations~$\vr$ of~$r$ and~$\vs$ of~$s$, we have~${|\vr \vee \vs| + |\vr \wedge \vs| \le |\vr| + |\vs|}$.
	Unless explicitly specified otherwise, we consider the \emph{(standard) order}~$|A,B|$ of a separation~$\{A,B\}$ of~$V$ as the cardinality of its \emph{separator}~$A \cap B$; this order function can easily be seen to be submodular.
	
	There are two special classes of separations which we will consider in this paper:
	a \emph{bipartition} of~$V$ is a separation of~$V$ whose sides are disjoint.
	We denote the set of all bipartitions of~$V$ by~$U_{bip}(V)$; note that~$U_{bip}(V)$ is again a universe of separations of~$V$.
	Since all bipartitions have order~$0$ with respect to our standard order function, we usually consider other order functions on~$U_{bip}(V)$ (see~\cite{TanglesSocial} for various examples).
	
	Another example of separations arises in graphs:
	a \emph{separation}~$\{A, B\}$ of a graph~$G = (V,E)$ is a separation of its vertex set~$V\!$ such that~$G$~has no edges between~$A\sm B$ and~$B\sm A$.
	The set of all separations of~$G$ then forms a universe of separations of~$V\!$.
	
	\subsection{Orientations}
	
	Let~$S$ be a set of unoriented separations of a set~$V$.
	Assigning to every~$s\in S$ either~$\vs$ or~$\sv$ is called \emph{orienting}~$S$ (or the~$s\in S$).
	So an \emph{orientation} of~$S$ is a set~$\tau$ of orientations of separations in~$S$ satisfying~$|\tau \cap \{\vs, \sv\}| = 1$ for every~$s\in S$.%
	\COMMENT{}
	Given an orientation~$\tau'$ of a subset~$S' \subseteq S$, we say that~$\tau'$ \emph{extends} to an orientation~$\tau$ of~$S$ if~$\tau' \subseteq \tau$.
	
	An orientation~$\tau$ of~$S$ is \emph{consistent} if there exist no two~$\vr, \vs \in \tau$ with ${\rv < \vs}$.
	If~$\tau$ is consistent and for every two distinct~$\vs, \vt \in \tau$, we have ${(\sv \wedge \tv) \notin \tau}$, then~$\tau$ is a \emph{profile} of~$S$.
	A profile of~$S$ is \emph{regular} if it does not contain any co-small separation.%
	\COMMENT{}
	Given a universe~$U$ of separations of~$V$, an order function on~$U$, and~$k \in \N$, we call a profile of the corresponding~$S_k$ a \emph{$k$-profile in~$U$}.
	
	Writing~$\vS := \{\vs, \sv \mid s \in S\}$ for the set of all orientations of separations in~$S$, let~$\F \subseteq 2^{\vS}$.%
	\COMMENT{}
	An orientation~$\tau$ of~$S$ is an \emph{$\F$-tangle} of~$S$ if~$\tau$ is consistent and~$\sigma \nsubseteq \tau$ for every~$\sigma \in \F$.
	Let~$\mathcal{T}$ be the set consisting of all sets ${\{(A_1,B_1), (A_2, B_2), (A_3, B_3)\}}$ of (not necessarily distinct) oriented separations in~$\vS$ with~$A_1 \cup A_2 \cup A_3 = V$, i.e.\ the supremum of the~$(A_i, B_i)$ is co-small.
	The~$\mathcal{T}$-tangles of~$S$ are called \emph{(abstract) tangles} of~$S$; they are examples of regular profiles of~$S$~\cite{AbstractTangles}.
	Given an order function on a universe~$U$ of separations of~$V$ and~$k \in \N$, we call a tangle of~$S_k \subseteq U$ a \emph{$k$-tangle in~$U$}.
	
	In the case that all separations in~$S$ are even bipartitions of~$V$, we will also consider~$\F^\ell$-tangles where~$\ell \in \R_{\ge 0}$.\footnote{We slightly deviate here from~\cite{AbstractTangles} in that our~$\F^\ell$ correspond to their~$\F^\ell_3$ and in that we consider~$\ell \in \R_{\ge 0}$ instead of~$\ell \in \N$.}
	Here,~$\F^\ell \subseteq 2^{\vS}$ consists of all sets $\{(A_1,B_1), (A_2, B_2), (A_3, B_3)\}$ of (not necessarily distinct) oriented bipartitions in~$\vS$ with~$|B_1 \cap B_2 \cap B_3| < \ell$.
	Note that~$\mathcal{T} = \F^1$ here since~$S$ consists of bipartitions of~$V$.

	\subsection{Point sets and functions inducing orientations} \label{subsec:Deciders}
	
	A \emph{weight function} on a set~$V\!$ is a map~$w$ from~$V\!$ to~$\mathbb{R}_{\ge 0}$.
	For a subset~$Z \subseteq V\!$ we write~$w(Z) = \sum_{v \in Z} w(v)$.
	If there exists~$v \in V\!$ with~$w(v) > 0$, then~$w$ is \emph{non-zero}.
	If~$w$ takes values in~$\{0,1\}$ only, then it can equivalently be formulated as an indicator function of the set~$X = X_w = w^{-1}(1)$ in that~$w(Z) = |X \cap Z|$ for every~$Z \subseteq V\!$; we shall use this equivalence freely throughout.
	
	For any weight function~$w$ on~$V$ and any separation~$\{A,B\}$ of~$V\!$, we have $w(B) - w(A) =  w(B \sm A) - w(A \sm B)$, a fact we shall use freely throughout.
	We say that~$w$ \emph{orients} a separation~$\{A,B\}$ of~$V\!$ \emph{as~$(A, B)$} if~$w(A) < w(B)$.
	More generally, $w$~\emph{orients} a set~$S$ of separations of~$V\!$ \emph{as} an orientation~$\tau$ of~$S$ if $w(A) < w(B)$ for all~$(A, B) \in \tau$.
	
	Conversely, let~$S$ be a set of separations of a set~$V$, and let~$\tau$ be an orientation of~$S$.
	If a weight function~$w$ on~$V$ orients~$S$ as~$\tau$, then we say that $w$ \emph{induces}~$\tau$ and all its elements.
	If~$w$ induces~$\tau$ and takes values in~$\{0,1\}$, then we say that the set~$X = w^{-1}(1)$ \emph{induces}~$\tau$.
	
	\medbreak
	
	Let us note some basic observations about orientations induced by functions.	
	First, let~$w$ be a weight function on~$V$ and let~$\lambda > 0$ be a scalar.
	If we \emph{scale~$w$ by~$\lambda$}, i.e.\ if we consider the weight function~$v \mapsto \lambda w(v)$ on~$V$, then this scaled weight function agrees with~$w$ on the sign of~$w(B)-w(A)$ for any separation~$\{A,B\}$ of~$V\!$.
	In particular, if an orientation~$\tau$ of a set~$S$ of separations of~$V$ is induced by a function~$w$, then for any given~$K > 0$ there exists a function~$w_K$ inducing~$\tau$ with~$w_K(B) - w_K(A) \ge K$ for all separations $(A,B) \in \tau$.
	This is because~$w_K$ can be chosen as an appropriate scaling of~$w$, i.e.\ by a factor~$\lambda\ge K / (\min_{(A,B) \in \tau} (w(B) - w(A)))$. 
	
	This fact directly implies that if an orientation~$\tau$ of a set~$S$ of separations of~$V$ is induced by a function~$w$, then there also exists a function inducing~$\tau$ which takes values in~$\Z_{\ge 0}$ instead of~$\R_{\ge 0}$.
	Indeed, there exists~$\varepsilon > 0$ such that~$w(B) - w(A) \geq \varepsilon$ for all~$(A, B) \in \tau$.
	Since~$\mathbb{Q}$ is dense in~$\R$, we can replace for every~$v \in V$ the scalar~$w(v) \in \R_{\ge 0}$ with a rational number~$w'(v) \in \mathbb{Q}_{\ge 0}$ which satisfies $|w(v) - w'(v)| < \varepsilon/|V|$.
	By construction, the resulting weight function~$w'$ on~$V$ still induces~$\tau$.
	Now we can scale~$w'$ by an appropriate~$\lambda \in \N$ to obtain the desired function which takes values in~$\Z_{\ge 0}$ and induces~$\tau$. 
	
	For the final observation in this section, let~$S$ be a set of separations of a set~$V$, and let~$\tau$ be an orientation of~$S$.
	Then a weight function on~$V\!$ induces~$\tau$ as soon as it induces the maximal elements of~$\tau$.
	We include a proof of this observation from~\cite{weighted_deciders_AIC} for the reader's convenience.
	
	\begin{LEM}\label{lem:MaximalSeparations}
		Let~$w$ be a weight function on a set~$V\!$, and let~$(A,B)$ and~$(C,D)$ be separations of~$V\!$ with~$(A,B) \le (C,D)$.
		Then~$w(B) - w(A) \ge w(D) - w(C)$.
		In particular,~$w$ induces~$(A,B)$ if it induces~$(C,D)$.
	\end{LEM}
	\begin{proof}
		Since~$(A,B) \le (C,D)$, we have~$A \subseteq C$ and~$B \supseteq D$.
		So~$w(A) \le w(C)$ and~$w(B) \ge w(D)$ as~$w$ is a weight function on~$V$.
		This directly implies that~$w(B) - w(A) \ge w(D) - w(C)$. 
		The `in particular'-part then follows immediately.
	\end{proof}

	\section{Resilience and locally induced orientations}\label{sec:Resilience}
	
	In this section we use the novel notion of \emph{resilience} to prove a sufficient criterion for an orientation of a set~$S$ of separations to be induced by a function. 
	After that, we further generalise the concept of resilience towards the notion of being `$k$-locally \hbox{$\ell$-induced}' which allows us to give a characterisation of those orientations of~$S$ which are induced by functions on the ground set~$V\!$.
	We begin this section by giving all the definitions around the concept of resilience.
	
	Let~$S$ be a set of separations of a set~$V$, and let~$k \in \N$.
	An orientation~$\tau$ of~$S$ is \emph{$k$-resilient} if no set of at most~$k$ separations in~$\tau$ has a co-small supremum.
	So~$\tau$ is~$k$-resilient if and only if for all sets~$\sigma \subseteq \tau$ of size at most~$k$, we have that~$\bigcup_{(A,B) \in \sigma} A \ne V$.
	If~$S$ consists only of bipartitions of~$V\!$, then this is equivalent to~$\bigcap_{(A,B) \in \sigma} B \ne \emptyset$ for all sets~$\sigma \subseteq \tau$ of size at most~$k$ since~$(V, \emptyset)$ is the only co-small bipartition of~$V\!$.
	Note that in order to determine whether~$\tau$ is $k$-resilient, it is always enough to consider sets~$\sigma$ of maximal elements of~$\tau$.\looseness=-1
	
	If~$\tau$ is $k$-resilient, then~$\tau$ is also~$k'$-resilient for every~$k' < k$.
	We call~$\tau$ \emph{infinitely resilient} if~$\tau$ is~$k$-resilient for all~$k \in \N$.
	The \emph{resilience} of~$\tau$ is the maximal~$k \in \N$ such that~$\tau$ is~$k$-resilient if such~$k$ exists, $0$~if~$\tau$ is not~$k$-resilient for any~$k \in \N$, and~$\infty$ otherwise.
	
	In addition to the examples on resilience given in the introduction, let us here illustrate the concept once more with a less extreme example.
	Consider the~$(n\times n)$-grid for some~$n \ge 5$, and let~$S$ be the set of all separations of this graph which have order at most~$4$.
	It is easy to see that the orientation~$\tau$ of~$S$ which is induced by the entire vertex set of the grid is a tangle.%
	\COMMENT{}
	Let us show that~$\tau$ is~$\Omega(n^2)$-resilient.
	Every element~$(A,B)$ of~$\tau$ satisfies~$|A|\le 10$; indeed, most satisfy~$|A|\le 5$.%
	\COMMENT{}
	Thus, any set of separations in~$\tau$ with a co-small supremum has at least~$n^2/10$ elements as the~$(n \times n)$-grid has precisely~$n^2$ vertices. 
	
	\medbreak
	
	Why can the notion of resilience help us with constructing a function that induces a given orientation?
	Consider an orientation~$\tau$ of a set~$S$ of separations of a set~$V\!$, and write~$\Max = \Max(\tau)$ for the set of maximal elements of~$\tau$.
	Let us see how resilience that is large in terms of~$|\Max|$ can help us build a $\tau$-inducing function.
	
	Assume that~$\tau$ is~$k$-resilient for some~$k \in \N$, and recall that~$\Max^{(k)}$ denotes the set of~$k$-element subsets of~$\Max$.
	Then the resilience of~$\tau$ implies that for every~$\Max' \in \Max^{(k)}$, there exists some~$v_{\Max'} \in V\!$ that is not contained in the small side of any separation in~$\Max'$; in particular, the set~$\{v_{\Max'}\}$ induces~$\Max'$.
	It seems natural to construct a function that induces~$\Max$ (and thus~$\tau$, by~\cref{lem:MaximalSeparations}) by combining all these local $\Max'$-inducing sets~$\{v_{\Max'}\}$, i.e.\ by assigning to each~${v \in V\!}$ as its weight the number of sets~$\Max' \in \Max^{(k)}$ with~$v_{\Max'} = v$.
	
	It turns out that the weight function~$w$ on~$V$ defined in this way need not in general be a function that induces~$\Max$.
	This is because each~$v_{\Max'}$, while adding its weight to the big sides of the separations in~$\Max'$, can also add weight to the small sides of separations in~$\Max \sm \Max'$.
	But as soon as~$k$ is large enough in that each fixed separation~$(A,B) \in \Max$ is contained in the majority of the sets in~$\Max^{(k)}$, which will happen as soon as~$\Max$ has more~$(k-1)$-subsets to form a~$k$-subset with~$(A,B)$ than it has~$k$-subsets not including~$(A,B)$, the orientation~$(A,B)$ of~$\{A,B\}$ will be induced by the majority of the sets~$\{v_{\Max'}\}$ that locally induce $\Max' \in \Max^{(k)}$.
	We can then deduce from this that~$w$ induces~$\Max$ and hence~$\tau$ by \cref{lem:MaximalSeparations}.
	
	More precisely, we have the following generalisation of~\cref{thm:ResilienceWeightedInformal} to arbitrary orientations~$\tau$:
	
	\begin{THM} \label{thm:ResilienceWeighted}
		Let~$S$ be a set of separations of a set~$V$, and let~$\tau$ be an orientation of~$S$.
		Let~$m$ be the number of maximal elements of~$\tau$.
		If~$\tau$ is \hbox{$k$-resilient} for some integer~$k > \frac{m}{2}$, then~$\tau$ is induced by a function on~$V\!$.
	\end{THM}
	
	We will formally obtain \cref{thm:ResilienceWeighted}, and hence \cref{thm:ResilienceWeightedInformal}, below as a corollary of the more general \cref{thm:DecidableWeighted}.
	But before we do so, let us first show that both \cref{thm:ResilienceWeightedInformal} and \cref{thm:ResilienceWeighted} are optimal with respect to the parameter~$k$ in~$k$-resilience. To show this we exhibit a suitable tangle that is not induced by a function, by using a more general version of the construction from the introduction.
	
	\begin{PROP} \label{prop:ResIsSharp}
		For all~$m, k \in \N$ with~$3 \le k\le \frac{m}{2}$, there exist a set~$V\!$, a sub\-modular order function~$| \cdot |$ on~$U_{bip}(V)$, and an~$m$-tangle~$\tau_{m,k}$ in~$U_{bip}(V)$ that has~$m$ maximal elements and is~$k$-resilient, but which is not induced by any function on~$V\!$.
	\end{PROP}
	
	\noindent An example of the tangles in \cref{prop:ResIsSharp} is given by a certain type of hypergraph edge tangles introduced in~\cite{weighted_deciders_AIC}.
	We now describe their construction, and then show that these tangles do indeed have all the desired properties.
	
	\begin{proof}[Proof of \cref{prop:ResIsSharp}]
		Let~$V = [m]^{(k)}$ consist of all~$k$-element subsets of~$[m]$.
		For every~$i \in [m]$ let~${V_i = \{X \in V \mid i \in X\}}$ be the set of all~$k$-element subsets of~$[m]$ containing~$i$.
		We assign to each bipartition~$\{A,B\}$ of~$V$ as its order~$|A,B|$ the number of sets~$V_i$ meeting both~$A$ and~$B$.
		This order function~$| \cdot |$ on~$U_{bip}(V)$ is easily seen to be submodular (see~\cite{weighted_deciders_AIC} for a formal proof).
		
		Every bipartition~$\{A, B\}$ of~$V$ of order less than~$m$ has precisely one side which contains~$V_i$ for some~$i \in [m]$.
		Indeed, one such side exists by the definition of the order function~$|\cdot|$, and since~$V_i \cap V_j \neq \emptyset$ for distinct~$i, j \in [m]$, this side is unique.
		Hence,
		\begin{equation*}
			\tau_{m,k} := \{ (A, B) \mid \{A, B\} \in S_m \text{ and } \exists\ i \in [m] \colon V_i \subseteq B\}
		\end{equation*}
		is a well-defined orientation of~$S_m \subseteq U_{bip}(V)$.
		The maximal elements of~$\tau_{m,k}$ are precisely the~$\vsi = (V \sm V_i, V_i)$ for~$i \in [m]$.
		So in order to see that~$\tau_{m,k}$ is indeed an~$m$-tangle in~$U_{bip}(V)$, it is enough to observe that~$\vsi \vee \vsj \vee \vsell$ is not co-small for any~${1\le i,j, \ell\le m}$.
		But since~$k \ge 3$, we always have~${V_i \cap V_j \cap V_{\ell} \ne \es}$.
		
		As shown in~\cite{weighted_deciders_AIC}, it is immediate from double counting (as in the example from the introduction) that~$\tau_{m,k}$ is not induced by any function on~$V\!$ for~${m \ge 6}$ and~$k \le \frac{m}{2}$.
		So it remains to check that~$\tau_{m,k}$ is~$k$-resilient.
		But this is immediate from the construction:
		for every collection~$\vsione, \dots, \vsik$ of~$k$ distinct maximal elements of~$\tau_{m,k}$, we have $\{i_1, \dots, i_k\} \in \bigcap_{j \in [k]} V_{i_j}$.
		Hence for every collection of at most~$k$ maximal elements of~$\tau_{m,k}$, the intersection of their big sides is non-empty.
	\end{proof}
	
	The~$\tau_{m,k}$ constructed in our proof of~\cref{prop:ResIsSharp} are abstract tangles, but the construction can easily be modified to find~$\F^\ell$-tangles for arbitrary~$\ell > 1$ with the same properties:
	instead of taking the~$k$-element subsets of~$[m]$ as the set~$V\!$, we can take~$V\!$ as the disjoint union of~$\lceil \ell \rceil$-element sets, one for every $k$-element subset of~$[m]$.
	
	\medbreak
	
	Before we proceed towards a proof of \cref{thm:DecidableWeighted}, our generalisation of \cref{thm:ResilienceWeighted}, let us briefly investigate in some more detail the above examples of~$\F^\ell$-tangles that are not induced by functions.
	Note that our construction of \hbox{$\F^\ell$-tangles} does not necessarily work when the value of~$\ell$ is not constant, but large in terms of~$|V|$, e.g.\ of size at least~$\varepsilon\,|V|$ for some constant~$\varepsilon > 0$.%
	\COMMENT{}
	The following proposition shows that there exists a sharp lower bound for those~$\varepsilon > 0$ for which~$\ell\ge \varepsilon\,|V|$ guarantees the existence of a function that induces the~\hbox{$\F^\ell$-tangle}.
	
	\begin{PROP}
		Let~$V\!$ be an~$n$-set, and let~$0<\varepsilon<1$.
		If~$\varepsilon\ge 1/8$, then every~$\F^{\varepsilon n}$-tangle~$\tau$ of a set~$S$ of bipartitions of~$V\!$ is induced by a function on~$V\!$; if~$\varepsilon>1/8$, then~$\tau$ is even induced by a subset of~$V\!$.
		
		Conversely, for every~$\varepsilon< 1/8$ there exist~$n \in \N$ and a set~$S$ of bipartitions of an~$n$-set such that some~$\F^{\varepsilon n}$-tangle of~$S$ is not induced by any function on~$V\!$.
	\end{PROP}
	\begin{proof}
		Let~$\tau$ be an~$\F^\ell$-tangle of a set~$S$ of bipartitions of a set~$V\!$ satisfying~${\ell\ge |V|/8}$.
		If all of~$V\!$ induces~$\tau$, then we are done; so suppose not.
		Then there exists~$(A_1,B_1) \in \tau$ with~$|B_1|\le |V|/2$.
		Again we are done if~$B_1$ induces~$\tau$.
		If this is not the case, then there is some~$(A_2,B_2) \in \tau$ with~$|B_1\cap A_2|\ge |B_1\cap B_2|$; in particular, we have~$|B_1\cap B_2|\le |V|/4$.
		
		It turns out that if~$\ell>|V|/8$, then~$B_1 \cap B_2$ has to induce~$\tau$, since otherwise there exists some~$(A_3,B_3) \in \tau$ such that
		\begin{equation*}
			|(B_1\cap B_2)\cap B_3|\le |(B_1\cap B_2)\cap A_3|.
		\end{equation*}
		This implies~$|(B_1\cap B_2)\cap B_3|\le |V|/8$, which contradicts the fact that~$\tau$ is an~$\F^\ell$-tangle. 
		
		If~$\ell=|V|/8$, then the same arguments as above produce a $\tau$-inducing set if at least one of the occurring inequalities is strict.
		So suppose that all the above inequalities are satisfied with equality.
		In particular, every~${(A,B) \in \tau}$ satisfies~$|A|\le|B|$.
		
		Using similar reasoning as above, we can obtain a function that induces~$\tau$.
		It will be enough to find a weight function~$w$ on~$V\!$ that induces the set~$\tau'\subseteq \tau$ consisting of all $(A,B)\in \tau$ with~$|A|=|B|$. 
		For a given function~$w$ inducing~$\tau'$, we can obtain a function inducing~$\tau$ by adding large enough constant weight to all the vertices in~$V\!$.
		
		Suppose there are~$(A,B), (C,D) \in \tau'$ with $|B\cap C|>|B\cap D|$.
		This yields $|B\cap D|<|V|/4$, which in turn implies the existence of a set inducing~$\tau$, by the same arguments as above. 
		Consequently, for every~$(A,B), (C,D) \in \tau'$, we have $|B\cap C|\le |B\cap D|$.
		
		Hence, the weight function~$w$ defined by counting for every~$v\in V\!$ the number of~$(A,B)\in \tau'$ with~$v\in B$ is a function inducing~$\tau'$: 
		given~${(C,D)\in \tau'}$, we have ${w(C)=\sum_{(A,B)\in \tau'}|B\cap C|}$ and ${w(D)=\sum_{(A,B)\in \tau'}|B\cap D|}$.
		As above, we have~${|B\cap C|\le|B\cap D|}$ for every~$(A,B)\in \tau'$, and as~$\tau$ is an~$\F^{\ell}$-tangle, we clearly have~$D \neq \es$ and hence~${|D\cap C|<|D\cap D|}$.
		This yields~$w(C)<w(D)$, and thus, $w$~induces~$\tau$.
		
		For the second part of the proposition, consider~$V$ and the~$m$-tangle~$\tau_{m,k}$ in~$U_{bip}(V)$ as constructed in our proof of \cref{prop:ResIsSharp} for some~${m\ge 2k\ge 6}$.
		Then for any three maximal elements of~$\tau_{m,k}$, their intersection contains exactly~$\binom{m-3}{k-3}$ elements of~$V\!$.
		In particular,~$\tau_{m,k}$ is an~$\F^{\ell}$-tangle for all~${\ell<\binom{m-3}{k-3}}$.
		
		Recall that~$|V|=\binom{m}{k}$.
		For~${k=m/2}$ we have~${\lim_{m\to \infty} \binom{m-3}{k-3}/\binom{m}{k}=1/8}$.
		Thus, we find for any~$\varepsilon<1/8$ some~$n = \binom{m}{k} \in \N$ such that the tangle~$\tau_{m,k}$ from \cref{prop:ResIsSharp} is an~$\F^{\varepsilon n}$-tangle of a set of bipartitions of an $n$-set that is not induced by any function on~$V\!$.
	\end{proof}
	
	Back to~\cref{thm:ResilienceWeighted}, recall that this theorem is sharp in terms of the parameter~$k$ in~$k$-resilience as shown by~\cref{prop:ResIsSharp}.
	However, the converse of \cref{thm:ResilienceWeighted} fails, i.e.\ not even every tangle induced by a set of points has high resilience compared to the number of its maximal elements.
	
	\begin{EX} \label{ex:DecidableNotResilient}
		Let~$V$ be a set of size~$n \ge 4$, and let~$S$ be the set of all bipartitions of~$V$ that have a side of size less than~$n/3$.%
		\COMMENT{}
		Let~$\tau$ be the orientation of~$S$ which orients every~$s \in S$ in such a way that its big side contains more elements of~$V$ than its small side.
		In particular, $V\!$~induces~$\tau$.	
		
		This orientation~$\tau$ of~$S$ is a tangle, since no three big sides of bipartitions in~$\tau$ have empty intersection.
		However, four big sides can, so the supremum of four bipartitions in~$\tau$ can be co-small.
		Thus,~$\tau$ has resilience~$3$.
		
		Now~$\tau$ has~$m = \binom{n}{\lceil (n/3)-1 \rceil}$ maximal elements, namely those bipartitions of~$V$ whose small side has maximum size.
		In particular, the resilience of~$\tau$ is low compared with~$m$, although~$\tau$ is induced by a function on~$V\!$, and even by the set~$V\!$. \looseness=-1
	\end{EX}
	
	It turns out that we can generalise the notion of resilience in a way which includes the tangle from the previous \cref{ex:DecidableNotResilient} without invalidating \cref{thm:ResilienceWeighted}.
	In fact, our more general notion leads to a more general result, \cref{thm:DecidableWeighted}, which actually characterises the orientations induced by functions.
	
	Let~$\tau$ be an orientation of a set of separations of a set~$V$, and let $\Max = \Max(\tau)$ be the set of maximal elements of~$\tau$.
	Our more general notion of resilience is based on our earlier observation that, if~$\tau$ is~$k$-resilient, then this provides for every $\Max' \in \Max^{(k)}$ a~singleton set~$\{v_{\Max'}\}$ that induces~$\Max'$.
	In the following definition we ask, instead of $k$-resilience, that for every $\Max' \in \Max^{(k)}$ there exists a function~$w_{\Max'}$ which induces~$\Max'$ and is not too badly wrong on the separations in~$\Max \sm \Max'$.
	
	More precisely, an orientation~$\tau$ of a set~$S$ of separations of a set~$V\!$ is \emph{$k$-locally\penalty-200\ $\ell$-induced} for given~$k \in \N$ and~$\ell \ge 0$ if for every set~$\Max' \subseteq \tau$ of size~$|\Max'| \le k$, there is a weight function~$w_{\Max'}$ on~$V\!$ satisfying
	\begin{enumerate}[\rm (i)]
		\item $\forall\ (A,B) \in \Max' \colon w_{\Max'}(B) - w_{\Max'}(A) \ge 1$; \label{LocDec1}
		\item $\forall\ (A,B) \in \tau \colon w_{\Max'}(A) - w_{\Max'}(B) \le \ell$. \label{LocDec2}
	\end{enumerate}
	
	Observe that by \cref{lem:MaximalSeparations} one needs only consider sets~$\Max' \subseteq \Max$ in the above definition where~$\Max = \Max(\tau)$ is the set of maximal elements of~$\tau$.
	In addition, we can equivalently strengthen the condition of~$|\Max'| \le k$ to~$|\Max'| = k$ (i.e.\ $\Max' \in \Max^{(k)}$), as long as~$\tau$ has at least~$k$ maximal elements.%
	\COMMENT{}
	
	The above definition is indeed a generalisation of our earlier notion of resilience, since a~$k$-resilient orientation~$\tau$ of~$S$ is~$k$-locally $1$-induced:
	for~${\Max' \subseteq \tau}$ with~$|\Max'| \le k$, take~$w_{\Max'}$ assigning~$1$ to a single element in~$V\sm\bigcup_{(A,B)\in \Max'} A$, which is non-empty since~$\tau$ is~$k$-resilient, and~$0$ to all other elements in~$V\!$.%
	\COMMENT{}
	
	If~$\tau$ is induced by a function on~$V\!$, then~$\tau$ is~$k$-locally $\ell$-induced for all~${k \in \N}$ and~$\ell \ge 0$.
	Indeed, as described in~\cref{subsec:Deciders}, there exists a function that induces~$\tau$ and satisfies~\eqref{LocDec1} for all separations in~$\tau$, and any function inducing~$\tau$ clearly satisfies~\eqref{LocDec2}.
	In particular, the tangle in~\cref{ex:DecidableNotResilient} is $k$-locally $\ell$-induced for all~$k \in \N$ and~$\ell \ge 0$.
	
	Here, then, is our generalisation of \cref{thm:ResilienceWeighted}:
	\begin{THM} \label{thm:DecidableWeighted}
		Let~$\tau$ be an orientation of a set~$S$ of separations of a set~$V\!$, and suppose that~$\tau$ has~$m$ maximal elements.
		Then~$\tau$ is induced by a function on~$V\!$ if and only if it is~$k$-locally $\ell$-induced for some $k \in \N$ and $\ell > 0$ with $k > \frac{m}{1+1/\ell}$.
	\end{THM}
	
	\noindent Since every~$k$-resilient orientation~$\tau$ is $k$-locally $1$-induced, \cref{thm:ResilienceWeighted}, and hence~\cref{thm:ResilienceWeightedInformal} as well, are direct corollaries of \cref{thm:DecidableWeighted}.
	
	\begin{proof}[Proof of \cref{thm:DecidableWeighted}]
		We have seen above that if~$\tau$ is induced by a function on~$V\!$, then it is $k$-locally \hbox{$\ell$-induced} for every $k \in \N$ and $\ell \ge 0$.
		In particular, $\tau$~is $m$-locally $\ell$-induced for every~$\ell > 0$, and we have~$m > \frac{m}{1+1/\ell}$ in this case.
		
		For the converse, recall that by~\cref{lem:MaximalSeparations} it is enough to show that the set ${\Max = \Max(\tau)}$ of maximal elements of~$\tau$ is induced by some function on~$V\!$. 
		If~${k\ge m}$ the statement is true immediately; so suppose for the following that~${k<m}$.%
		\COMMENT{}
		We construct a function~$w$ inducing~$\Max$ as follows: 
		for every~${\Max' \in \Max^{(k)}}$ let~$w_{\Max'}$ be a weight function on~$V$ as in the definition of `$k$-locally $\ell$-induced'. 
		Then we combine all these~$w_{\Max'}$ to define the weight function~$w$ on~$V$ as
		\begin{equation*}
			w: V \rightarrow \mathbb{R}_{\ge 0}, w(v) = \sum_{\Max' \in \Max^{(k)}} w_{\Max'}(v).
		\end{equation*}
		We show that $w$ induces~$\Max$, as desired.
		
		For an arbitrary separation~$\vs = (A,B) \in \Max$, let~$\Max^{(k)}_{\vs}$ consist of all~${\Max' \in \Max^{(k)}}$ containing~$\vs$.
		Then~$\Max^{(k)}_{\vs}$ has size~$\binom{m-1}{k-1}$ since~$\vs$ is contained in~$\binom{m-1}{k-1}$ many sets~${\Max' \in \Max^{(k)}}$.
		Similarly,~$\Max^{(k)} \sm \Max^{(k)}_{\vs}$ has size~$\binom{m-1}{k}$.
		Therefore, \eqref{LocDec1}~in the definition of `$k$-locally $\ell$-induced' yields
		\begin{equation*}
			\sum_{\Max' \in \Max^{(k)}_{\vs}} w_{\Max'}(A) \le \sum_{\Max' \in \Max^{(k)}_{\vs}} (w_{\Max'}(B) - 1) = \sum_{\Max' \in \Max^{(k)}_{\vs}} w_{\Max'}(B) - \binom{m-1}{k-1}.
		\end{equation*}
		Similarly, we obtain by~\eqref{LocDec2} that
		\begin{equation*}
			\sum_{\Max' \in \Max^{(k)} \sm \Max^{(k)}_{\vs}} w_{\Max'}(A) \le \sum_{\Max' \in \Max^{(k)} \sm \Max^{(k)}_{\vs}} (w_{\Max'}(B)\ +\ \ell) = \sum_{\Max' \in \Max^{(k)} \sm \Max^{(k)}_{\vs}} w_{\Max'}(B)\ +\ \ell\ \cdot\ \binom{m-1}{k}.
		\end{equation*}
		These inequalities combine to	
		\begin{equation*}
			\begin{split}
				w(A) & = \sum_{\Max' \in \Max^{(k)}} w_{\Max'}(A) \\
				& = \sum_{\Max' \in \Max^{(k)}_{\vs}} w_{\Max'}(A) + \sum_{\Max' \in \Max^{(k)} \sm \Max^{(k)}_{\vs}} w_{\Max'}(A) \\
				& \le \sum_{\Max' \in \Max^{(k)}_{\vs}} w_{\Max'}(B) - \binom{m-1}{k-1} + \sum_{\Max' \in \Max^{(k)} \sm \Max^{(k)}_{\vs}} w_{\Max'}(B) + \ell \cdot \binom{m-1}{k} \\
				& = w(B) - \binom{m-1}{k-1} + \ell \cdot \binom{m-1}{k}.
			\end{split}
		\end{equation*}
		Now since~$k > \frac{m}{1+1/\ell}$ and~$k<m$, we have that~$\ell<\frac{k}{m-k}$ and thus 
		\begin{equation*}
			\ell \cdot \binom{m-1}{k} < \binom{m-1}{k-1},
		\end{equation*}
		as~$\binom{m-1}{k}=\frac{(m-1)!}{k!(m-1-k)!}$ and~$\binom{m-1}{k-1}=\frac{(m-1)!}{(k-1)!(m-k)!}$ differ precisely by the factor~$\frac{k}{m-k}$.
		This then implies $w(A) < w(B)$, so $w$ induces~$\vs = (A,B)$.
		Thus, since~$\vs \in \Max$ was arbitrarily chosen, $w$~induces~$\Max$ and hence~$\tau$, by \cref{lem:MaximalSeparations}.
	\end{proof}

	\section{Orientations and duality of set separations} \label{sec:Duality}
	
	In this section we present a second characterisation of those orientations of a set of separations that are induced by a function on the ground set: a~character\-isation in terms of a duality between separation systems.
	As an unexpected corollary, we obtain an independent second proof of \cref{thm:ResilienceWeighted}.
	
	The duality of separation systems, which was introduced in~\cite{TanglesSocial} and first studied in~\cite{ASSduality, HypergraphHomology}, is defined for set separations as follows.
	Let~$S$ be a set of separations of a set~$V\!$, and let~$\sigma$ be an orientation of~$S$.
	For every~$v \in V\!$, the~sets\looseness=-1
	\begin{align*}
		C_\sigma(v) &:= \{\{A, B\} \in S \mid (A, B) \in \sigma, v \in A\} \text{ and}\\
		D_\sigma(v) &:= \{\{A, B\} \in S \mid (A, B) \in \sigma, v \in B\}
	\end{align*}
	form the sides of the separation~$\{C_\sigma(v), D_\sigma(v)\}$ of~$S$.
	The map ${\varphi_\sigma: V \to U(S)}$ then associates with~$v \in V$ the separation~$\varphi_\sigma(v) := \{C_\sigma(v), D_\sigma(v)\}$ of~$S$.
	We write $V_\sigma := \varphi_\sigma(V)$ for the \emph{dual set of separations of~$S$ with respect to~$\sigma$}.%
	\COMMENT{}
	
	Suppose now that~$\varphi_\sigma$ is injective.\footnote{The map~$\varphi_\sigma$ is injective if and only if for every two distinct elements~$v, v' \in V$, there exist two separations in~$S$ for one of which~$v$ and~$v'$ are contained in the same side and for the other in different sides. If this is not the case for~$v, v' \in V$, then $v$ and~$v'$ `carry the same information' about how the separations in~$S$ separate~$V$ and can hence be seen as redundant.}
	Then the dual set~$V_\sigma$ of separations of~$S$ with respect to~$\sigma$ has a natural \emph{default orientation~$\tau_\sigma$} which orients every ${\varphi_\sigma(v) = \{C_\sigma(v), D_\sigma(v)\} \in V_\sigma}$ as~$(C_\sigma(v), D_\sigma(v))$, i.e.\ its big side contains those $\{A, B\} \in S$ with $(A, B) \in \sigma$ and~$v \in B$.
	It is easy to see that the dual set of separations of~$V_\sigma$ with respect to~$\tau_\sigma$ is again~$S$ with default orientation~$\sigma$.
	In simple terms, `dualising the dual yields the primal'~\cite{ASSduality}.
	
	If~$\varphi_\sigma$ is injective, we can ask whether the existence of a function on~$V\!$ that induces the orientation~$\sigma$ of~$S$ relates to any property of the default orientation~$\tau_\sigma$ of the dual set~$V_\sigma$ of separations of~$S$ with respect to~$\sigma$.
	It turns out that it does, and it does so in an intriguing way: $\sigma$~is induced by a function on~$V\!$ if and only if every non-zero weight function~$w'$ on~$S$ induces some separation in~$\tau_\sigma$.
	More generally, we have the following theorem:
	
	\begin{THM} \label{thm:WeightedLPDuality}
		Let~$S$ be a set of separations of a set~$V\!$, and let~$\sigma$ be an orientation~of~$S$.
		Then the following two assertions are equivalent:
		\begin{enumerate}[\rm (i)]
			\item There exists a function on~$V\!$ that induces~$\sigma$. \label{DThm1}
			\item For every non-zero weight function~$w'$ on~$S$, there exists $v \in V$ such that $w'(C_\sigma(v)) < w'(D_\sigma(v))$.\label{DThm2}
		\end{enumerate}
		In particular, if~$\varphi_\sigma$ is injective and~$\tau_\sigma$ is the default orientation of the dual set~$V_\sigma$ of separations of~$S$ with respect to~$\sigma$, then~$\sigma$ is induced by a function on~$V\!$ if and only if every non-zero weight function~$w'$ on~$S$ induces a separation in~$\tau_\sigma$. \looseness=-1
	\end{THM}
	
	\noindent The proof of \cref{thm:WeightedLPDuality} will be done purely in terms of linear algebra and can be followed without any further knowledge about our particular duality of separation systems described above.
	The key tool in our proof will be the following variant of Farkas's Lemma (see, e.g., \cite[6.~Theorem]{FarkasLemmaEquiv}\footnote{Our version of Farkas's Lemma follows from \cite[6.~Theorem]{FarkasLemmaEquiv} by applying their theorem to~$A=Q^T$ and~$-b$ instead of~$b$.}).
	
	\newpage
	
	\begin{LEM}[Farkas's Lemma] \label{lem:Farkas}
		Let~$Q \in \R^{n \times \ell}$ and~$b \in \R^\ell$.
		Then exactly one of the following two assertions holds:
		\begin{enumerate}[\rm (i)]
			\item There exists~$x \in \R_{\ge 0}^n$ with~$Q^T x \ge b$.\label{Farkas1}
			\item There exists~$y \in \R_{\ge 0}^\ell$ with~$Qy \le 0$ and~$b^T y > 0$.\label{Farkas2} 
		\end{enumerate}
	\end{LEM}
	
	\begin{proof}[Proof of \cref{thm:WeightedLPDuality}]
		Fix enumerations~$V = \{v_1, \dots v_n\}$ and~$S = \{s_1, \dots, s_\ell\}$, and let~${\vsj = (A_j, B_j)  \in \sigma}$ for~$j \in [\ell]$.
		Using these enumerations, we shall, for the course of this proof, identify a weight function~$w$ on~$V\!$ with a vector~${x = x(w)}$ in~$\R_{\ge 0}^n$ and a weight function~$w'$ on~$S$ with a vector $y = y(w')$ in $\R_{\ge 0}^\ell$.
		
		Let us define a matrix~$Q = Q(\sigma) \in \R^{n \times \ell}$ via
		\begin{equation*}
			Q_{ij} =
			\begin{cases}
				~~1, & v_i \in B_j \sm A_j;\\
				~~0, & v_i \in A_j \cap B_j;\\
				-1, & v_i \in A_j \sm B_j.
			\end{cases}
		\end{equation*}
		Recall from the definition of~$\varphi_\sigma$ that~$C_i := C_\sigma(v_i)$ (respectively~${D_i := D_\sigma(v_i)}$) consists of all those~$\{A_j, B_j\} \in S$ with~$v_i \in A_j$ (respectively~$v_i \in B_j$).
		So given a weight function~$w'$ on~$S$, we obtain for~$y = y(w') \in \R_{\ge 0}^\ell$ and all~$i \in [n]$ that 
		\begin{equation*}
			(Qy)_i = \sum_{B_j \ni v_i} y_j - \sum_{A_j \ni v_i} y_j = w'(D_i) - w'(C_i).
		\end{equation*}
		So a non-zero weight function~$w'$ on~$S$ is \emph{not} as in \cref{thm:WeightedLPDuality}\,\eqref{DThm2} if and only if~$Qy \le 0$ 
		for~$y = y(w') \in \R_{\ge 0}^\ell$ where~$\le$~is meant coordinate-wise. 
		
		Similarly, let~$w$ be a weight function on~$V\!$ and~$x = x(w) \in \R_{\ge 0}^n$.
		Then we compute for all~$j \in [\ell]$ that
		\begin{equation*}
			(Q^T x)_j = \sum_{v_i \in B_j} x_i - \sum_{v_i \in A_j} x_i = w(B_j) - w(A_j).
		\end{equation*}
		So a weight function~$w$ on~$V\!$ induces~$\sigma$ if and only if we have~$Q^T x > 0$ for ${x = x(w) \in \R_{\ge 0}^n}$.
		Recall from \cref{subsec:Deciders} that~$\sigma$ is induced by some function on~$V\!$ if and only if there exists a function~$w$ with~${w(B) - w(A) \ge 1}$ for all~$(A,B) \in \sigma$.
		Hence, $\sigma$~is induced by some function on~$V\!$ (as in \cref{thm:WeightedLPDuality}\,\eqref{DThm1}) if and only if~$Q^T x \ge \mathbf{1}$ where~$x = x(w) \in \R_{\ge 0}^n$ for some non-zero weight function~$w$ on~$V\!$ and~$\mathbf{1}$ is the constant~$1$ vector in~$\R^\ell$.
		
		The result then follows by applying \cref{lem:Farkas} to~$Q$ and~$b=\mathbf{1}\in \R^\ell$, and denoting~$x = x(w)$ in~\cref{lem:Farkas}\,\eqref{Farkas1} and~$y = y(w')$ in~\cref{lem:Farkas}\,\eqref{Farkas2}.
		The `in particular'-part is immediate from the definition of~$V_\sigma$ and~$\tau_\sigma$.
	\end{proof}
	
	As mapping a separation in~$S$ to its orientation in~$\sigma$ is a bijection between~$S$ and~$\sigma$, we could equivalently define the weight function~$w'$ in~\cref{thm:WeightedLPDuality}\,\eqref{DThm2} on the orientation~$\sigma$ of~$S$; for notational simplicity we will freely switch between these two definitions in what follows.
	
	By \cref{lem:MaximalSeparations}, an orientation~$\sigma$ of~$S$ is induced by a function on~$V\!$ if and only if its set of maximal elements is.
	So applying \cref{thm:WeightedLPDuality} to the set~$\Max = \Max(\sigma)$ of maximal elements of~$\sigma$ and its underlying set of unoriented separations,	we have the following corollary.
	
	\begin{COR}  \label{cor:WeightedLPDualityMax}
		Let~$S$ be a set of separations of a set~$V\!$.
		Let~$\sigma$ be an orientation of~$S$, and let~$\Max = \Max(\sigma)$ be the set of maximal elements of~$\sigma$.
		Then the following two assertions are equivalent:
		\begin{enumerate}[\rm (i)]
			\item There exists a function on~$V\!$ that induces~$\sigma$. \label{DCor1}
			\item For every non-zero weight function~$w'$ on~$\Max$, there exists~$v \in V\!$ with \label{DCor2}
			\begin{equation*}
				\displaystyle\sum_{\substack{(A,B) \in \Max \\\text{with } v \in A}} w'((A,B)) < \displaystyle\sum_{\substack{(A,B) \in \Max \\\text{with } v \in B}} w'((A,B)).\eqno\qed
			\end{equation*}
		\end{enumerate}
	\end{COR}
	
	As an illustration of the power of \cref{cor:WeightedLPDualityMax} and hence \cref{thm:WeightedLPDuality}, let us re-prove \cref{thm:ResilienceWeighted}, which asserts that highly resilient orientations are induced by functions on the ground set.
	
	\begin{PROP} \label{prop:FResilienceWeighted2}
		Let~$S$ be a set of separations of a set~$V$, and let~$\sigma$ be an orientation of~$S$ with~$m$ maximal elements.
		If~$\sigma$ is~$k$-resilient for some integer~$k > \frac{m}{2}$, then~$\sigma$ is induced by a function on~$V\!$.
	\end{PROP}
	
	\begin{proof}
		Let~$\Max = \Max(\sigma)$ be the set of maximal elements of~$\sigma$.
		We apply \cref{cor:WeightedLPDualityMax} to~$\sigma$ in that we consider an arbitrary non-zero weight function~$w'$~on~$\Max$ and show that case~\eqref{DCor2} in \cref{cor:WeightedLPDualityMax} holds.
		Let~$\Max' \subseteq \Max$ consist of those~$k$ separations in~$\Max$ which have the highest weight with respect to~$w'$.
		Since we have~${k > \frac{m}{2}}$ by assumption, this immediately yields~$w'(\Max') > w'(\Max \setminus \Max')$.
		
		Now~$\sigma$ is~$k$-resilient, so there exists some~$v \in V\!$ which is not contained in the small side of any separation in~$\Max'$.
		By the choice of~$\Max'$, this~$v$ satisfies
		\begin{equation*}
			\displaystyle\sum_{\substack{(A,B) \in \Max \\\text{with } v \in B}} w'((A,B))
			\ge w'(\Max') > w'(\Max \setminus \Max')
			\ge \displaystyle\sum_{\substack{(A,B) \in \Max \\\text{with } v \in A}} w'((A,B)).
		\end{equation*}
		Thus, \cref{cor:WeightedLPDualityMax}~\eqref{DCor2} holds for~$w'$, and since~$w'$ was arbitrarily chosen, this implies that~$\sigma$ is induced by a function on~$V\!$.
	\end{proof}

	\section{Extendable tangles are induced by point sets} \label{sec:Extendable}
	
	Let~$S$ be a set of separations of a set~$V\!$ of `points', and let~$\tau$ be an orientation of~$S$.
	In \cref{sec:Resilience} we analysed various properties of~$\tau$ which ensure that $\tau$~is induced by a function on~$V\!$.
	All these properties required us to consider large subsets of~$\tau$ instead of the usual triples which are required for the definition of a tangle of~$S$.
	In particular, all the notions considered above may be viewed as strengthenings of the triple condition in the definition of a tangle, i.e., we give a stronger condition that an orientation needs to satisfy in order to be a tangle that is induced by some function on~$V\!$.
	
	But how can we guarantee the existence of a function inducing a tangle~$\tau$ of~$S$ if we do not want to strengthen the definition of a tangle in the above sense?
	We know that there exist tangles that are not induced by any function on the ground set of the separations they orient (see e.g.\ \cref{prop:ResIsSharp}).
	So instead of looking for a function inducing~$\tau$ itself, we may try to find one that induces some~$\tau' \subseteq \tau$.
	Ideally, we can do this in such a way that the function inducing~$\tau'$ is still, in some sense, related to the original tangle~$\tau$.
	
	Given an order function on a universe~$U$ of separations of a set~$V$, one natural such subset of a~$k$-tangle~$\tau$ in~$U$, say, consists of all separations in~$\tau$ of order less than some~$k'< k$.
	In other words, we would like to obtain, given a~$k$-tangle~$\tau$ in~$U$, a function  on~$V\!$ that induces the~$k'$-tangle~$\tau'\subseteq \tau$ in~$U$.
	
	One way in which we could try to achieve this consists in proving the following:
	there exists a function~$f: \N \to \N$ with~$f(k) \ge k$ for all~$k \in \N$ such that if a~$k$-tangle~$\tau'$ in~$U$ extends to a~$f(k)$-tangle~$\tau$ in~$U$, then~$\tau'$ is induced by a function on~$V\!$, or even by a subset of~$V\!$.%
	\COMMENT{}
	In this case, we may view the function~$w$ inducing~$\tau'$ as an approximation of a function inducing its extension~$\tau$~-- although~$w$ will in general not orient~$S_{f(k)} \subseteq U$ as~$\tau$. 
	
	Consider for example the~$m$-tangle~$\tau_{m,k}$ in~$U_{bip}(V)$ constructed in \cref{prop:ResIsSharp} for some~$3 \le k\le \frac{m}{2}$.
	This tangle~$\tau_{m,k}$ is not induced by any function on~$V\!$, but if we consider only the separations of order less than~$\frac{m}{2}$ in this example, then they are even induced by a set:
	$V\!$~orients all the separations in~$U_{bip}(V)$ of order less than~$\frac{m}{2}$ in the same way as~$\tau_{m,k}$.
	
	This leads us to the question of whether tangles which extend to tangles of twice their order are always induced by functions or even sets, i.e., whether~${f(k) := 2k}$ is suitable.
	We show that this is indeed the case in that $k$-profiles in~$U$ which extend to regular~$2k$-profiles in~$U$ are induced by subsets of~$V\!$~-- as long as we work in the universe $U = U(V)$ of separations of a set~$V\!$ equipped with our standard order function on~$U$ which assigns to a separation the cardinality of its separator as its order.
	Recall that, for this order function, all~$k$-profiles in~$U$ are induced by functions on~$V\!$~\cite{weighted_deciders_AIC}, but those as above are even induced by subsets of~$V\!$:%
	\COMMENT{}
	
	\begin{THM} \label{thm:InducedTangleDeciderSet}
		Let~$U = U(V)$ be the universe of all separations of a set~$V$ and let~$|\cdot|$ be the standard order function on~$U$.
		If~$\tau'$ is a~$k$-profile in~$U$ for some~$k \in \N$ that extends to a regular~$2k$-profile~$\tau$ in~$U$, then~$\tau'$ is induced by a set~$X \subseteq V\!$ of size at least~$2k$.
	\end{THM}
	
	\begin{proof}
		If~$V$ has less than~$2k$~elements, then there exists no regular~\hbox{$2k$-profile~$\tau$} in~$U$, since~$(V,V) \in \tau$ contradicts its regularity.
		Thus, the theorem always holds for~$|V| < 2k$, and we may assume~$|V| \ge 2k$ in what follows.
		
		Our desired set~$X$ inducing~$\tau'$ will be the \emph{interior}~$\bigcap_{(A,B) \in \sigma} B$ of a star~$\sigma$ contained in~$\tau$.
		Let us first show that such interiors cannot be too small.
		
		\begin{CLAIM}\label{lem:induced_star_size}
			The interior of any star contained in a regular~$2k$-profile in~$U$ has at least~$2k$ elements.
		\end{CLAIM}
		
		\begin{proof}
			Suppose not, let~$\tau$ be a regular~$2k$-profile in~$U$, and let~$\sigma \subseteq \tau$ be a star whose interior~$X = \bigcap_{(A,B) \in \sigma} B$ has size~$|X| < 2k$. 
			Note that~$\sigma$ is non-empty as the interior of the empty star is the whole set~$V$ which by assumption has size at least~$2k$.
			
			Let us write~$\sigma=\{(A_1,B_1), \dots, (A_\ell,B_\ell)\}$. 
			We claim that for any~$i\le \ell$ we have~${|(A_1,B_1)\vee \dots \vee (A_i,B_i)| < 2k}$.
			By definition, we have
			\begin{equation*}
				|(A_1,B_1)\vee \dots \vee (A_i,B_i)|=|(A_1\cup \dots \cup A_i)\cap (B_1\cap \dots\cap B_i)|.
			\end{equation*}
			Since~$\sigma$ is a star, we have~${(A_1\cup \dots \cup A_i)\subseteq B_j}$ for every~$j>i$.
			So in particular, we have
			\begin{equation*}
				(A_1\cup \dots \cup A_i)\cap (B_1\cap \dots\cap B_i)\subseteq B_1\cap \dots\cap B_\ell = X.
			\end{equation*}
			Therefore,~${|(A_1,B_1)\vee \dots \vee (A_i,B_i)|\le |X| < 2k}$. 
			Since~$\tau$ is a profile, it follows inductively that~${(A_1,B_1)\vee \dots \vee (A_i,B_i)\in \tau}$ for every~$i\le \ell$.
			In particular, we have~${(Y, X) = (A_1,B_1)\vee \dots \vee (A_\ell,B_\ell) \in \tau}$ where~$Y = \bigcup_{(A,B) \in \sigma} A$.
			
			Since~$|X| < 2k$, the separation~$\{X,V\}$ has order~$< 2k$ and hence an orientation in~$\tau$.
			By the regularity of~$\tau$, this orientation must be~$(X,V)$ because~$(V,X)$ is co-small.
			But this leads to a contradiction since this would imply~$(Y,X) \vee (X,V) = (V,X) \in \tau$ as~$\tau$ is a profile.
		\end{proof}
		
		Let~$\sigma \subseteq \tau$ be a star whose interior~$X = \bigcap_{(C,D) \in \sigma} D$ is of smallest size among all stars contained in~$\tau$.
		By~\cref{lem:induced_star_size} we have~$|X|\ge 2k$.
		We claim that~$X$ induces~$\tau'$.
		
		To prove this, we show that~$|X \cap A| < k$ for every~$(A,B) \in \tau'$.
		Since we have~${|X| \ge 2k}$, this immediately implies that~$X$ induces~$\tau'$.
		So suppose for a contradiction that there exists~$(A, B) \in \tau'$ with ~$|X \cap A| \ge k$, and assume that~$(A, B)$ has minimal order among all such separations in~$\tau'$.
		Note that~$(A,B)$ may be induced by~$X$.
		
		Now for every~${(C,D) \in \sigma}$, the separation~$(A \cap D, B \cup C)$ has at least the order of~$(A, B)$.
		Indeed, we have~$X \subseteq D$ by construction, and therefore we get~$|(A \cap D) \cap X| = |A \cap X| \ge k$.
		Moreover, if~$(A \cap D, B \cup C)$ would have order less than~$(A,B)$, then~$(A \cap D, B \cup C)$ would be contained in~$\tau$ (since~$\tau$ is a profile) and hence in~$\tau' \subseteq \tau$ because~$(A,B)$ has order less than~$k$.
		Then, however,~$(A \cap D, B \cup C)$ would contradict the minimal choice of~$(A,B)$.
		
		By the submodularity of the standard order function,~${(B \cap C, A \cup D)}$ has order at most~$|C,D|$.
		Since~${(B \cap C, A \cup D) \le (C,D) \in \tau}$ and~$\tau$ is consistent, we thus have~$(B \cap C, A \cup D) \in \tau$.
		Hence, the star
		\begin{equation*}
			\hat{\sigma} := \{(A, B)\} \cup \{(B \cap C, A \cup D) \mid (C, D) \in \sigma\}
		\end{equation*}
		is contained in~$\tau$.
		
		We claim that the interior~$\hat{X}$ of~$\hat{\sigma}$ is smaller than~$X$ contradicting the choice of~$\sigma$. 
		Indeed, by definition, we have
		\begin{equation*}
			\hat{X} = B \cap \bigcap_{(C,D) \in \sigma} (A \cup D)
			= (A \cap B) \cup (B \cap X)
			= ((A \cap B) \sm X) \cup (B \cap X).
		\end{equation*}
		Since~$(A,B)$ is a separation of~$V$, the set~$X \subseteq V$ is the disjoint union of~$B\cap X$ and~$(A \cap X) \sm B$.
		So we are done if
		\begin{equation*}
			|(A\cap B) \sm X| < |(A \cap X) \sm B|.
		\end{equation*}
		Let~$h = |A \cap B \cap X|$.
		Since~$|A \cap X| \ge k$, we have~$|(A\cap X) \sm B| \ge k - h$.
		However, we have~$(A,B) \in \tau$, so~$|A \cap B| < k$ and hence~$|(A\cap B) \sm X| < k - h$, completing the proof.
	\end{proof}
	
	Our proof of~\cref{thm:InducedTangleDeciderSet} heavily relies on the assumption that the order of a separation in~$U(V)$ is given by the size of its separator.
	We do not know if a similar result holds for other or even all submodular order functions on~$U(V)$.
	
	\begin{PROB}
		Let~$V$ be a set, and consider any submodular order function on~$U(V)$.
		Is it true that if~$\tau'$ is a~$k$-profile in~$U(V)$ for some~$k \in \N$ which extends to a regular~$2k$-profile~$\tau$ in~$U(V)$, then~$\tau'$ is induced by some subset of~$V$?
		What happens for other universes of separations of~$V$ such as~$U_{bip}(V)$?
	\end{PROB}

	\section*{Acknowledgements}

	The third author gratefully acknowledges support by doctoral scholarships of the Studienstiftung des deutschen Volkes and the Cusanuswerk -- Bisch\"{o}fliche Studienförderung.
	
	\bibliography{collective}

\begin{thebibliography}{10}

\bibitem{SARefiningInessParts}
S.~Albrechtsen.
\newblock Refining trees of tangles in abstract separation
  systems{~I:~I}nessential parts.
\newblock ar{X}iv:2302.01808, 2023.

\bibitem{DiestelBook16noEE}
R.~Diestel.
\newblock {\em {Graph Theory}}.
\newblock Springer, 5th edition, 2017.

\bibitem{ASS}
R.~Diestel.
\newblock Abstract separation systems.
\newblock {\em Order}, 35:157--170, 2018; arXiv:1406.3797.

\bibitem{TanglesSocial}
R.~Diestel.
\newblock Tangles in the social sciences: a new mathematical model to identify
  types and predict behaviour.
\newblock arXiv:1907.07341, 2019.

\bibitem{HypergraphHomology}
R.~Diestel.
\newblock Homological aspects of oriented hypergraphs.
\newblock arXiv:2007.09125, 2021.

\bibitem{TangleBook}
R.~Diestel.
\newblock {\em Tangles: indirect clustering in the empirical sciences}.
\newblock Monograph, in preparation.

\bibitem{ASSduality}
R.~Diestel, J.~Erde, C.~Elbracht, and M.~Teegen.
\newblock Duality and tangles of set partitions.
\newblock {\em J. Combinatorics}, to appear; arXiv:2109.08398.

\bibitem{TextTangles}
R.~Diestel, J.~Erde, and K.~Krishnareddy.
\newblock Tangles: a novel way to classify texts.
\newblock In preparation.

\bibitem{AbstractTangles}
R.~Diestel, J.~Erde, and D.~Wei{\ss}auer.
\newblock Structural submodularity and tangles in abstract separation systems.
\newblock {\em J. Combin. Theory (Series~A)}, 167C:155--180, 2019;
  ar{X}iv:1805.01439.

\bibitem{ProfilesNew}
R.~Diestel, F.~Hundertmark, and S.~Lemanczyk.
\newblock Profiles of separations: in graphs, matroids, and beyond.
\newblock {\em Combinatorica}, 39(1):37--75, 2019; arXiv:1110.6207.

\bibitem{TangleTreeAbstract}
R.~Diestel and S.~Oum.
\newblock Tangle-tree duality in abstract separation systems.
\newblock {\em Advances in Mathematics}, 377:107470, 2021; arXiv:1701.02509.

\bibitem{TangleClusteringWeakStrong}
C.~Elbracht, D.~Fioravanti, S.~Klepper, J.~Kneip, L.~Rendsburg, M.~Teegen, and
  U.~von Luxburg.
\newblock Clustering with tangles: Algorithmic framework and theoretical
  guarantees.
\newblock arXiv:2006.14444, 2020.

\bibitem{CanonicalToTSubmodular}
C.~Elbracht and J.~Kneip.
\newblock A canonical tree-of-tangles theorem for structurally submodular
  separation systems.
\newblock {\em Combinatorial Theory}, 1(5), 2021; arXiv:2009.02091.

\bibitem{weighted_deciders_AIC}
C.~Elbracht, J.~Kneip, and M.~Teegen.
\newblock Tangles are decided by weighted vertex sets.
\newblock {\em Advances in Comb.}, 2020:9, 2020; arXiv:1811.06821.

\bibitem{FiniteSplinters}
C.~Elbracht, J.~Kneip, and M.~Teegen.
\newblock Trees of tangles in abstract separation systems.
\newblock {\em J. Combin. Theory (Series~A)}, 180:105425, 2021;
  arXiv:1909.09030.

\bibitem{InfiniteSplinters}
C.~Elbracht, J.~Kneip, and M.~Teegen.
\newblock Trees of tangles in infinite separation systems.
\newblock {\em Math. Proc. Camb. Phil. Soc.}, pages 1--31, 2021;
  arXiv:2005.12122.

\bibitem{ToTfromTTD}
C.~Elbracht, J.~Kneip, and M.~Teegen.
\newblock Obtaining trees of tangles from tangle-tree duality.
\newblock {\em J. Combinatorics}, 13:251--287, 2022; arXiv:2011.09758.

\bibitem{BranchDecMatroids}
J.~Geelen, B.~Gerards, N.~Robertson, and G.~Whittle.
\newblock Obstructions to branch-decomposition of matroids.
\newblock {\em J.~Combin.\ Theory (Series B)}, 96:560--570, 2006.

\bibitem{FarkasLemmaEquiv}
C.~Perng.
\newblock On a class of theorems equivalent to {{Farkas}}'s lemma.
\newblock {\em Applied Mathematical Sciences}, 11:2175--2184, 2017.

\bibitem{GMX}
N.~Robertson and P.~Seymour.
\newblock Graph minors. {X}. {O}bstructions to tree-decomposition.
\newblock {\em Journal of Combinatorial Theory, Series~B}, 52:153--190, 1991.

\end{thebibliography}
	
\end{document}